\documentclass[a4paper, 12pt]{amsart}
\usepackage[utf8]{inputenc}
\usepackage[T2A]{fontenc}
\usepackage[english]{babel}

\usepackage{amssymb}

\newtheorem{lemma}{Lemma}
\newtheorem{theorem}{Theorem}
\newtheorem{example}{Example}
\newtheorem{definition}{Definition}

\def\hm#1{#1\nobreak\discretionary{}{\hbox{\ensuremath{#1}}}{}}

\def\Z{{\mathbb Z}}
\def\N{{\mathbb N}}

\def\eps{\varepsilon}
\def\phi{\varphi}

\def\dec{\mathop{\mathrm{dec}}}

\def\le{\leqslant}
\def\ge{\geqslant}
\def\tt{{\mathfrak t}}
\def\TT{{\mathfrak T}}

\usepackage[unicode]{hyperref}

\begin{document}

\title[Parity of the number of primes and sublinear summations]{Parity of the number of~primes in a given interval and algorithms of the sublinear summation}
\author{Andrew V. Lelechenko}
\address{I.~I.~Mechnikov Odessa National University}
\email{1@dxdy.ru}


\begin{abstract}
Recently Tao, Croot and Helfgott \cite{tao2012} invented an algorithm to determine the parity of the number of primes in a given interval in $O(x^{1/2-c+\eps})$ steps for some absolute constant~$c$. We propose a slightly different approach, which leads to the implicit value of~$c$.

To achieve this aim we discuss the summation of multiplicative functions, developing sublinear algorithms and proving several general theorems.
\end{abstract}

\maketitle

\section{Introduction}

How many operations are required to find  any prime~$p > x$ (not necessary the closest) for given $x$? 

A direct approach is to apply AKS primality test~\cite{agrawal2004}, which was im\-pro\-ved by Lenstra and Pomerance \cite{lenstra2011} to run in time~$O(\log^{6+\eps} x)$, on  consecutive integers starting with $x$. Such method leads to an algorithm with average complexity~$O(\log^{7+\eps} x)$, because in average we should run AKS $\log x$ times before a next prime encounters.

But in the worst case available estimates of the complexity are much bigger; they depend on upper bounds of the gaps between primes. The best currently known result on the gaps between primes is by Baker, Harman and Pintz~\cite{baker2001}: for large enough $x$ there exists at least one prime in the interval 
$$[x,x+x^{0.525+\eps}].$$
Thus we obtain that the worst case of an algorithm may need up to
$$O(x^{0.525+\eps}) \gg x^{1/2}$$
operations.

One can propose another algorithm, which is distinct from the pointwise testing. Suppose that there is a test, which allows to determine whether a given interval $[a,b]\subset[x,2x]$ contains at least one prime in~$A(x)$ operations. Then (starting with interval $[x,2x]$) we are able to find a~prime $p>x$ in $A(x)\log x$ operations using a dichotomy.

A test to determine whether a given interval contains at least one prime can be built atop Lagarias---Odlyzko formula for $\pi(x)$ \cite{lagarias1987}, which provides an al\-go\-rithm with $O(x^{1/2+\eps}) \gg x^{1/2}$ complexity. See \cite{tao2012} for more detailed discussion.

In \cite{tao2012} Tao, Croot and Helfgott offer a hypothesis that there exists an al\-go\-rithm to compute~$\pi(x)$
in $O(x^{1/2-c+\eps})$ operations, where $c>0$ is some absolute constant. This implies that a prime $p>x$ can be found in~$O(x^{1/2-c+\eps}) \ll x^{1/2}$ steps. Authors prove the following weaker theorem \cite[Th. 1.2]{tao2012}.

\begin{theorem}[Tao, Croot and Helfgott, 2012]
There exists an absolute constant $c>0$, such that one can (deterministically) decide whether a~given interval $[a,b]$ in $[x,2x]$ of length at most $x^{1/2+c}$ contains an \emph{odd number} of primes in time $O(x^{1/2-c+o(1)})$.
\end{theorem}

The aim of our paper is to prove the following result. 

\begin{theorem}\label{th:polymath2}
Let $[a,b]\subset[x,2x]$, $b-a\le x^{1/2+c}$, $c$ is an arbitrary constant such that $0<c\le 1/2$. Then a parity of 
$\#\{ p\in[a,b] \}$ can be determined in time $$O(x^{\max(c,7/15)+\eps}).$$
\end{theorem}

In Section \ref{s:general} we discuss a general approach to sum up multiplicative functions on given intervals. In Section \ref{s:fast} we consider cases when such summation can be done in a sublinear time. Finally, in Section \ref{s:proof} Theorem \ref{th:polymath2} is proven.

\section{The general summation algorithm}\label{s:general}

Consider the summation 
$$\sum_{n\le x} f(n),$$
where $f$ is a multiplicative function, from the complexity's point of view.

\bigskip

Generally speaking, a property of the multiplicativity does not impose signi\-fi\-cant restrictions on pointwise computational complexity. Multiplicative functions can be both easily-computable (e. g., $f(n)\hm=n^k$ for every~$k$) and hardly-computable: e. g.,
$$ f(p^\alpha)=\left\{ 
\begin{array}{cl}
2, & \quad\text{if there are $p^\alpha$ consecutive zeroes in digits of $\pi$} \\
1, & \quad\text{otherwise,}
\end{array} \right. $$

Luckily the vast majority of multiplicative functions, which have applications in the number theory, are relatively easily-computable. 

\begin{definition}
A multiplicative function $f$ is called {\em easily-computable,} if for any prime $p$, integer $\alpha>0$ and real $\eps>0$ the value of $f(p^\alpha)$ can be computed in time $O(p^\eps \alpha^m)$ for some absolute constant $m$, depending only on~$f$.
\end{definition}

\begin{example}\rm
The (two-dimensional) divisor function $\tau_2(p^\alpha) = \alpha+1$, the (two-dimensional) unitary divisor function $\tau_2^*(p^\alpha) = 2$, the totient function~$\phi(p^\alpha) = p^\alpha-p^{\alpha-1}$, the sum-of-divisors function $\sigma(p^\alpha) \hm= (p^{\alpha+1}-1)/(p-1)$, the Möbius function $\mu(p^\alpha)=[\alpha<2](-1)^\alpha$ are examples of easily-computable multiplicative functions for any $m>0$.
\end{example}

\begin{example}\rm
Let $a(n)$ be the number of non-isomorphic abelian groups of order $n$. Then $a(p^\alpha)=P(\alpha)$, where $P(n)$ is a number of partitions of $n$. It is known \cite[Note I.19]{flajolet2009}, that $P(n)$ is computable in $O(n^{3/2})$ operations. Thus function $a(n)$ is an easily-computable multiplicative function with~$m=3/2$. 

The number of rings of $n$ elements is known to be multiplicative, but no explicit formula exists currently for $\alpha\ge4$. See OEIS \cite{oeis2011} sequences A027623, A037289 and A037290 for further discussions.
\end{example}

\begin{example}\rm
The Ramanujan tau function $\tau_R$ is a rare example of an important number-theoretical multiplicative function, which is not easily-com\-pu\-table. The best known result is due to Charles~\cite{charles2006}: a value of $\tau_R(p^\alpha)$ can be computed by $p$ and $\alpha$ in $O(p^{3/4+\eps} + \alpha)$ operations.
\end{example}

Surely pointwise product and sum of easily-computable functions are also easily-computable ones. The following statement shows that the Dirichlet con\-vo\-lu\-tion 
$$
(f\star g)(n) = \sum_{d\mid n} f(d) g(n/d)
$$
also saves a property of easily-computability.

\begin{lemma}\label{l:convolution-computability}
If $f$ and $g$ are easily-computable multiplicative functions, then
$$h:=f\star g$$
is also easily-computable.
\end{lemma}
\begin{proof}
By definition of easily-computable functions there exists $m$ such that $f(p^\alpha)$ and $g(p^\alpha)$ can be both computed in $O(p^\eps \alpha^m)$ time. 

By definition of the Dirichlet convolution
$$ 
h(p^\alpha) = \sum_{a=0}^\alpha f(p^a) g(p^{\alpha-a}).
$$
This means that computation of $h(p^\alpha)$ requires
$$
\sum_{a=0}^\alpha O(p^\eps a^m + p^\eps (\alpha-a)^m) \ll p^\eps \alpha^{m+1}
$$
operations.
\end{proof}

\bigskip

Firstly, consider a trivial summation algorithm: calculate values of function pointwise and sum them up. For an easily-computable multiplicative function the majority of time will be spend on the factoring numbers from~$1$ to $x$ one-by-one. But no polynomial-time factoring algorithm is currently known; the best algorithms (e. g., GNFS \cite{lenstra1993}) have complexities about 
$$
\exp\left((c+\eps)(\log n)^{\frac{1}{3}}(\log \log n)^{\frac{2}{3}}\right),
$$
which is very expensive. 
 
\bigskip

We propose a faster general method like the sieve of Eratosthenes. We shall refer to it as to {\em Algorithm M.}

\bigskip

{\bf Algorithm M.} Consider an array $A$ of length $x$, filled with integers from~$1$ to $x$, and an array $B$ of the same length, filled with~1. Values of $f(n)$ will be computed in the corresponding cells of $B$. 

For each prime $p\le\sqrt{x}$ cache values of $f(p),f(p^2),\ldots,f(p^{\lfloor \log x / \log p \rfloor})$ and take integers 
$$
k=p,2p,3p,\ldots,\lfloor x/p \rfloor p
$$ 
one-by-one; for each of them determine $\alpha$ such that $p^\alpha \parallel k$ and replace~$A[k]$ by~$A[k]/p^\alpha$ and $B[k]$ by $B[k]\cdot f(p^\alpha)$.  

After such steps cells of $A$ contain 1 or primes $p > \sqrt{x}$. So for each~$n$ such that~$A[n]\ne 1$ multiply $B[n]$ by $f(A[n])$. 

Now array $B$ contains computed values of $f(1),\ldots,f(n)$. Sum up its cells to end the algorithm. 

\bigskip

Algorithm M can be encoded in pseudocode as it is shown in Listing~\ref{list:M}. 

\begin{figure}
\flushleft
\par $sum(\textit{ff}, x) = $
\par ~~ $\Sigma = 0$
\par ~~ $A \gets \{k\}_{k=1}^x$
\par ~~ $B \gets \{1\}_{k=1}^x$
\par ~~ for prime $p \le \sqrt{x} $ 
\par ~~ ~~ $F \gets \{\textit{ff}(p, \alpha)\}_{\alpha=1}^{\log x/\log p}$
\par ~~ ~~ for $k \gets p,2p,\ldots,\lfloor x/p \rfloor p $
\par ~~ ~~ ~~ $\alpha \gets \max \{a \mid p^a|k \} $
\par ~~ ~~ ~~ $A[k] \gets A[k]/p^\alpha$
\par ~~ ~~ ~~ $B[k] \gets B[k]\cdot F[\alpha]$
\par ~~ for $n \gets 1,\ldots,x$
\par ~~ ~~ if $A[n]\ne1  \Rightarrow B[n] \gets B[n]\cdot \textit{ff}(n, 1) $
\par ~~ for $n \gets 1,\ldots,x$
\par ~~ ~~ $\Sigma \gets \Sigma + B[n]$
\par ~~ return $\Sigma$
{
\renewcommand\figurename{Listing}
\caption{Pseudocode of Algorithm M. Here $\textit{ff}(p,\alpha)$ stands for the routine that effectively computes $f(p^\alpha)$.}
\label{list:M}
}
\end{figure}

Note that (similarly to the sieve of Eratosthenes) instead of the continuous array of  length $x$ one can manipulate with the set of arrays of length~$\Omega(\sqrt x)$. Inner cycles can be run independently of the order; they can be paralleled easily. Also one can compute several easily-computable functions simultaneously with a slight modification of Algorithm M.

\begin{lemma}
\label{l:sieve-compl}
If $f$ is an easily-computable multiplicative function then Al\-go\-rithm M runs in time $O(x^{1+\eps})$.
\end{lemma}
\begin{proof}
The description of Algorithm M shows that its running time is asymptotically lesser than
$$
\sum_{p\le \sqrt{x}} p^\eps \sum_{\alpha \le \log x / \log p} \alpha^m 
+ \sum_{p\le \sqrt{x}} {x\over p} 
+ \sum_{\sqrt{x} < p \le x} p^\eps
\ll 
x^{1+\eps}.
$$
\end{proof}

\section{The fast summation}\label{s:fast}

\begin{definition}
We say that function $f$ {\em sums up with the deceleration~$a$,} if the function~$F(x)\hm=\sum_{n\le x} f(n)$ can be computed in
$O(x^{a+\eps})$ time. 

Denote the deceleration of $f$ as $\dec f$. Notation $\dec f = a$ means exactly that there exists a method to sum up function $f$ with the deceleration $a$ (not necessarily there is no faster method).
\end{definition}

\begin{example}\rm
Lemma \ref{l:sieve-compl} shows that any easily-computable multiplicative function sums up with the deceleration $1$.
\end{example}

\begin{example}\label{ex:k-powers}\rm
Function $f(n)=n^k$, $k\in\mathbb{Z}_+$, sums up in time $O(1)$, because there is an explicit formula for $F(x)$ using Bernoulli numbers. Thus its de\-ce\-le\-ra\-tion is equal to $0$. Note that Dirichlet series of $f$ is~$\zeta(s-k)$, including case~$\zeta(s)$ when $k=0$.

One can check that the same can be said about $f(n)=\chi(n) n^k$, where~$\chi$ is an arbitrary multiplicative character modulo $m$. We just split $F(x)$ into $m$ sums of powers of the elements of arithmetic progressions. In this case Dirichlet series equals to $L(s-k,\chi)$.
\end{example}

\begin{example}\label{ex:charact-k-powers}\rm
The characteristic function of $k$-th powers, $k\in\mathbb{N}$, 
sums up in~$O(1)$ trivially, so its deceleration equals to $0$. Dirichlet series of such function is $\zeta(ks)$. 

Consider now $f$ such that $f(n^k)=\chi(n)$ and $f(n)=0$ otherwise, where $\chi$ is a multiplicative character. Then 
$$
\sum_{n=1}^\infty {f(n)\over n^s} = L(ks, \chi).
$$
Such function $f$ also sums up in $O(1)$, because
$F(x) = \sum_{n\le x^{1/k}} \chi(n)$
(see Example \ref{ex:k-powers}).

Generally, if function $f$ has Dirichlet series ${\mathcal F}(s)$ and  function $g$ has Dirichlet series ${\mathcal F}(ks)$ 
then 
$
\dec g = (\dec f) / k.
$
\end{example}

\begin{example}\label{ex:mu}\rm
Consider Mertens function $M(x):=\sum_{n\le x} \mu(n)$. 
In \cite{deleglise1996b} an algorithm of computation of $M(x)$ is proposed with time complexity~$O(x^{2/3}\*\log^{1/3}\log x)$ and memory consumption $O(x^{1/3}\*\log^{2/3}\log x)$. We obtain~$\dec\mu \hm= 2/3$. 

Note that Dirichlet series of $\mu$ equals to $1/\zeta(s)$.

One can see that a function $\mu_k$ such that $\mu_k(n^k)=\mu(n)$ and~$\mu_k(n)\hm=0$ otherwise sums up with the deceleration $2/(3k)$. Its Dirichlet series is $1/\zeta(ks)$.
\end{example}

\begin{example}\label{ex:tau2}\rm
In \cite{tao2012} an algorithm of computation of 
$T_2(x) := \sum_{n\le x} \tau_2(n)$ in $O(x^{1/3+\eps})$ time is described. Another algorithm with the same complexity may be found in \cite{sladkey2012}, accompanied with detailed account and pseudocode implementation. Thus~$\dec\tau_2=1/3$. 
\end{example}

\begin{theorem}\label{th:qsum1}
Let $f$ and $g$ be two easily-computable multiplicative functions, which sums up with decelerations $a := \dec f$ and $b := \dec g$ such that~$a\hm+b<2$. Then $h:=f\star g$ sums up with the deceleration 
$$ \dec h = {1-ab \over 2-a-b}.$$
\end{theorem}

\begin{proof}
Let
$$ F(x):=\sum_{n\le x} f(n), \quad G(x):=\sum_{n\le x} g(n), \quad H(x):=\sum_{n\le x} h(n).$$
By definition of the Dirichlet convolution
$$
H(x)=\sum_{n\le x} \sum_{d_1 d_2 =n } f(d_1)g(d_2)=
\sum_{ d_1d_2\le x} f(d_1)g(d_2). 
$$
Rearrange items:
$$
\sum_{ d_1d_2\le x} =
\sum_{ \scriptstyle d_1\le x^c \atop \scriptstyle  d_2\le {x/ d_1} } +
\sum_{ \scriptstyle  d_1\le {x/ d_2} \atop \scriptstyle  d_2\le x^{1-c} }
- \sum_{ \scriptstyle  d_1\le x^c \atop \scriptstyle  d_2\le x^{1-c} },
$$
where an absolute constant $c\in(0,1)$ will be defined below in \eqref{eq:c-defined}. 
Now
\begin{equation}
\tag{1}
\label{eq:qsum1}
H(x)= \sum_{d\le x^c} f(d)G\left({x\over d}\right)+
\sum_{d\le x^{1-c}} g(d)F\left({x\over d}\right)-F(x^c)G(x^{1-c}).
\end{equation}
As far as we can calculate $f(1),\ldots,f(x^c)$ with Algorithm M in~$O(x^{c+\eps})$ steps,
we can compute the first sum at the right side of~\eqref{eq:qsum1} in time
\begin{multline*}
O(x^{c+\eps}) + \sum_{d\le x^c} O\left({x\over d}\right)^{b+\eps} \ll x^{b+\eps} \sum_{d\le x^c} d^{-b-\eps} 
\ll \\ \ll
x^{b+\eps} x^{c(1-b-\eps)}
\ll x^{c+b(1-c)+\eps}.
\end{multline*}
Similarly the second sum can be computed in $O(x^{1-c+ac+\eps})$ operations. The last item of~\eqref{eq:qsum1} can be computed in time $O(x^{ac+\eps}+x^{b(1-c)+\eps})$. 

It remains to select $c$ such that 
$ c+b(1-c) \hm= 1-c+ac$. Thus
\begin{equation}
\tag{2}
\label{eq:c-defined}
c={1-b \over 2-a-b},
\end{equation}
which implies the deceleration~$ (1-ab)/(2-a-b)$.
\end{proof}

\begin{example}\rm
Function $\sigma_k(n)$ maps $n$ into the sum of $k$-th powers of its divisors. Thus $\sigma_k(n)=\sum_{d | n} d^k$, which is the Dirichlet convolution of~$f(n)=n^k$ and ${\bf 1}(n)=1$. So Example \ref{ex:k-powers} and Theorem \ref{th:qsum1} shows that~$ \dec \sigma_k = 1/2$.
\end{example}

\begin{example}\rm
Consider $r(n)=\#\{(k,l) \mid k^2+l^2=n\}$. It is well-known that~$r(n)/4$ is a multiplicative function, and ${1\over4}R(x):=\sum_{n\le x}r(n)/4$ is the number of integer points in the first quadrant of the circle of radius $\sqrt x$. Then~$R(x)$ can be naturally computed in $O(x^{1/2})$ steps, so~$\dec r = 1/2$.

Dirichlet series of $r(n)/4$ equals to $\zeta(s) L(s,\chi_4)$, where $\chi_4$ is the single non-principal character modulo 4. This representation shows that~$r(\cdot)/4 = \chi_4 \star {\bf 1}$. Thus Example \ref{ex:k-powers} together with Theorem \ref{th:qsum1} gives us another way to estimate the deceleration of $r$.
\end{example}

\begin{example}\rm
By Möbius inversion formula for the totient function we have 
$$
\phi(n)=\sum_{d \mid n} d\mu(n/d).
$$ 
This representation implies that $\dec\phi=3/4$ (see Example~\ref{ex:mu} for $\dec\mu$). Jor\-dan's totient functions have the same deceleration, because 
$$ J_k(n) = \sum_{d\mid n} d^k \mu(n/d). $$
\end{example}

\begin{theorem}\label{th:qsum2}
Let $f$ be an easily-computable multiplicative function. Consider
$$f_k := \underbrace{f\star\cdots\star f}_{k\text{~factors}}. $$
Then
$$
\dec f_k = 1-{1-\dec f \over k}.
$$
\end{theorem}
\begin{proof}
Follows from iterative applications of Lemma \ref{l:convolution-computability} and Theorem \ref{th:qsum1} and from the identities
$$
{1-a^2\over 2-2a} = 1-{1-a\over 2},
$$
$$
{1-a(k+a-1)/k \over 2 - 1 + (1-a)/k - a}
= 1 - {1-a \over k+1}.
$$
\end{proof}

\begin{example}\label{ex:tauk}\rm
For the multidimensional divisor function $\tau_k$ representations
\begin{eqnarray*}
\tau_{2k}&=&\underbrace{\tau_2\star\ldots\star\tau_2}_{k \text{~factors}},
\\
\tau_{2k+1}&=&\underbrace{\tau_2\star\ldots\star\tau_2}_{k \text{~factors}} \star {\bf 1}
\end{eqnarray*} 
imply that by Example \ref{ex:tau2} and Theorem \ref{th:qsum2} 
function $\tau_{2k}$ sums up with the deceleration $1-2/(3k)$, and $\tau_{2k+1}$ with the deceleration~$1-2/(3k+2)$.

In other words 
\begin{equation}
\tag{3}
\label{eq:dec-tauk}
\dec \tau_k = \begin{cases}
1-4/(3k), & k \text{~is~even}, \\
1-4/(3k+1), & k \text{~is~odd}. \\
\end{cases}
\end{equation}

Considering 
$$\tau_{-k} = \underbrace{\mu\star\cdots\star\mu}_{k \text{~factors}},  $$
we obtain by Example~\ref{ex:mu} and Theorem \ref{th:qsum2} that $\dec\tau_{-k}=1-1/(3k)$.
\end{example}

\bigskip

Theorems~\ref{th:qsum1} and~\ref{th:qsum2} cannot provide the deceleration lower than $1/2$ even in the best case. To overcome this barrier we should develop better instruments. 

\begin{theorem}\label{th:qqsum}
Let $f$ and $g$ be two easily-computable multiplicative functions, which sums up with decelerations $a := \dec f$ and $b := \dec g$ such that~$a+b<2$. Let 
\begin{equation}
\tag{4}
\label{eq:convolution-generalized}
h(n) := \sum_{d_1^{k_1} d_2^{k_2} = n} f(d_1) g(d_2).
\end{equation}
Then $h$ sums up with the deceleration 
$$ \dec h = {1-ab \over (1-a)k_2+(1-b)k_1}.$$
\end{theorem}
\begin{proof}
Following the outline of the proof of Theorem~\ref{th:qsum1} we obtain identity
\begin{multline*}
H(x) = 
\sum_{d\le x^{c/k_1}} f(d) G\left( \sqrt[k_2]{x/ d^{k_1}} \right)
+ 
\sum_{d\le x^{(1-c)/k_2}} g(d) F\left(\sqrt[k_1]{x/d^{k_2}}\right)
-\\-
F(x^{c/k_1}) G(x^{(1-c)/k_2}).
\end{multline*}
Thus we need $y(x)$ operations to calculate $H(x)$, where
\begin{multline*}
y(x) 
\ll 
\sum_{d\le x^{c/k_1}} \left( {x\over d^{k_1}} \right)^{b/k_2}
+ 
\sum_{d\le x^{(1-c)/k_2}} \left( {x\over d^{k_2}} \right)^{a/k_1} 
+ \\ +
x^{ac/k_1} + x^{b(1-c)/k_2}
\ll \\ \ll
x^{b/k_2+(1-bk_1/k_2)\cdot c/k_1} 
+ 
x^{a/k_1 + (1-ak_2/k_1)\cdot(1-c)/k_2} 
+ \\ +
x^{ac/k_1} + x^{b(1-c)/k_2}.
\end{multline*}
Substitution 
$$
c = {(1-b)k_1 \over (1-a)k_2+(1-b)k_1}
$$ 
completes the proof.
\end{proof}

In terms of Dirichlet series identity \eqref{eq:convolution-generalized} means that
$$
{\mathcal H}(s) = {\mathcal F}(k_1s) {\mathcal G}(k_2s)
$$
where 
$$
{\mathcal F}(s) = \sum_{n=1}^\infty {f(n)\over n^{s}}, 
\quad
{\mathcal G}(s) = \sum_{n=1}^\infty {g(n)\over n^{s}}, 
\quad
{\mathcal H}(s) = \sum_{n=1}^\infty {h(n)\over n^{s}}.
$$

One can prove (similarly to Lemma \ref{l:convolution-computability}) that convolutions of form \eqref{eq:convolution-generalized} save a property of the easily-computability.

\begin{example}\label{th:qsum3}\rm
Function $\tau_2^*$ sums up with the deceleration $7/15$, because
$$
\tau^*_2(n)=\sum_{d^2 \mid n} \mu(d) \tau_2(n/d^2).
$$
\end{example}

\begin{example}\rm
As soon as
$$
\tau_2^2(n) = \sum_{d^2\mid n} \mu(d) \tau_4(n/d^2),
$$
we obtain $\dec\tau_2^2=5/9$.
\end{example}

\bigskip

The discussion in Examples \ref{ex:k-powers}, \ref{ex:charact-k-powers}, \ref{ex:mu} leads to the following general statement.

\begin{theorem}\label{th:general}
Let $f$ be a multiplicative function such that 
\begin{equation}\label{eq:zeta-product}
\tag{5}
\sum_{n=1}^\infty {f(n)\over n^s} = 
\prod_{m=1}^{M_1} \zeta(k_m s)^{\pm1} 
\prod_{m=1}^{M_2} z_m(l_m s-n_m)
,
\end{equation}
where each of $z_m$ is either $\zeta$ or $L(\cdot,\chi)$, $M_1$, $M_2$, $k_m, l_m, n_m \in \N$. Then~$f$ sums up in sublinear time: its deceleration is strictly less than $1$.
\end{theorem}

Theorem \ref{th:general} clearly shows that the concept of fast summation can be easily generalized over various quadratic fields. Following theorem is an example of such kind of results.

\begin{theorem}
Consider the ring of Gaussian integers $\Z[i]$. 
Let 
$$\tt_k\colon \Z[i]\hm\to\Z$$ 
be a $k$-dimensional divisor function on this ring. Let 
$$
\TT_k(x) := \sum_{N(\alpha)\le x} \tt_k(\alpha),
$$
where $N(a+ib) = a^2+b^2 $. Then $\TT_k(x)$ can be computed in sublinear time.
\end{theorem}
\begin{proof}
It is well-known that
$$
{1\over4} \sum_{\alpha\in\Z[i]} {\tt_k(\alpha) \over N^s(\alpha)}
= \zeta^k(s) L^k(s,\chi_4) = \sum_{n=1}^\infty {f(n) \over n^s},
$$
where 
$$
f(n) := \sum_{N(\alpha)=n} \tt_k(\alpha).
$$
But by Theorem \ref{th:qsum2}
$$
\dec \underbrace{\chi_4\star\cdots\star\chi_4}_{k\text{~factors}} = 1 - 1/k.
$$
By \eqref{eq:dec-tauk} we obtain that for even $k$ 
$$ 
\dec f = {1-(1-1/k)\bigl(1-4/(3k)\bigr) \over {1/k} + {4/(3k)}} = 1 - {4\over 7k}
$$ and for odd $k$ 
$$
\dec f = {1-(1-1/k)\bigl(1-4/(3k+1)\bigr) \over {1/k} + {4/(3k+1)}} = 1 - {4\over 7k+1}.
$$
\end{proof}

\section{Proof of the Theorem \ref{th:polymath2}}\label{s:proof}

The proof follows the outline of the proof of~\cite[Th. 1.2]{tao2012}, but uses improved bound for the complexity of the computation of 
$$
T_2^*(x) := \sum_{n\le x} \tau_2^*(n).
$$

\begin{proof}
Trivially we have
$$
\sum_{a\le n\le b} \tau_2^*(n) = T_2^*(b)-T_2^*(a-1).
$$
As soon as $\tau_2^*(n) = 2^{\omega(n)}$, where $\omega(n) = \sum_{p\mid n} 1$, 
all summands in the left side are divisible by 4, beside those, which corresponds to $n=p^j$. Moving to the congruence modulo 4, we obtain
$$
2\sum_{j=1}^{O(\log x)} \# \left\{ p \in\left[a^{1/j},b^{1/j}\right] \right\} \equiv T_2^*(b)-T_2^*(a-1) \pmod 4. 
$$
As far as $a>x$ and $b-a\le O(x^{1/2+c})$, then for $j>1$
interval~$\left[a^{1/j},b^{1/j}\right]$ con\-tains~$O(x^c)$ elements; thus all such summands can be computed in~$O(x^{c+\eps})$ steps using AKS primality test~\cite{agrawal2004}. The right side of the congruence is com\-pu\-table in $O(x^{7/15+\eps})$ operations due to Example~\ref{th:qsum3}. 

The discussion above shows that the desired quantity
\begin{multline*}
\# \bigl\{p\in[a,b]\bigr\} \equiv { T_2^*(b)-T_2^*(a-1) \over 2}
- 
\\
-\sum_{j=2}^{O(\log x)} \# \left\{ p \in\left[a^{1/j},b^{1/j}\right] \right\} \pmod 2 \end{multline*}
can be computed in $O(x^{\max(c,7/15)+\eps})$ steps.
\end{proof}

\section{Conclusion}

Further development of algorithms of the sublinear sum\-ma\-tion (e.~g., summation of $\mu$ in arithmetic progressions) will lead to the generalization of Theorem \ref{th:general} over broader classes of functions. Also one can investigate summation of $f$ such that its Dirichlet series is an infinite, but sparse product of form \eqref{eq:zeta-product}.

\bibliographystyle{ugost2008s}
\bibliography{parity}

\end{document}